\documentclass{amsart}
\usepackage{amssymb,latexsym}
\usepackage{amsfonts}
\usepackage{amsmath}
\usepackage{enumerate}
\usepackage{mathrsfs}
\usepackage{chngpage}
\usepackage[colorlinks,linkcolor=blue,anchorcolor=blue,citecolor=blue]{hyperref}

\newtheorem{theorem}{Theorem}[section]
\newtheorem{lemma}[theorem]{Lemma}

\theoremstyle{definition}

\newtheorem{corollary}[theorem]{Corollary}
\newtheorem{estimate}[theorem]{Estimate}
\newtheorem{proposition}[theorem]{Proposition}

\theoremstyle{remark}
\newtheorem{remark}[theorem]{Remark}

\numberwithin{equation}{section}

\newcommand\F{{\mathbb F}}
\newcommand\Q{{\mathbb Q}}

\newcommand\Z{{\mathbb Z}}

\newcommand\NN{{\mathcal{N}}}

\begin{document}

\title[Pairing-friendly elliptic curves]{Heuristics of the Cocks-Pinch method}

\author{Min Sha}
\address{Institut de Mathematiques de Bordeaux, Universite Bordeaux 1, 33405 Talence Cedex, France}
\email{shamin2010@gmail.com}
\thanks{The author is supported by the China Scholarship Council.}


\subjclass[2010]{Primary 14H52, 11T71, 11G20.}



\keywords{Pairing-friendly elliptic curve, Cocks-Pinch method, Bateman-Horn conjecture, pairing-friendly field}

\begin{abstract}
We heuristically analyze the Cocks-Pinch method by using the Bateman-Horn conjecture. Especially, we present the first known heuristic which suggests that any efficient construction of pairing-friendly elliptic curves can efficiently  generate such curves over pairing-friendly fields, naturally including the Cocks-Pinch method. Finally, some numerical evidence is given.
\end{abstract}

\maketitle



\section{Introduction}

\subsection{Motivation}
Mainly inspired by the following pioneering works: three-party one-round key agreement \cite{Joux2000}, identity-based encryption
\cite{Boneh2003,Sakai2000}, short signature scheme \cite{Boneh2004}, easing the cryptographic applications of pairings \cite{Verheul2001} and
efficient computation of pairings associated to elliptic curves \cite{Miller2004}, there has been a flurry of activity in the design and analysis
of cryptographic protocols by using pairings on elliptic curves. More in-depth studies of pairing-based cryptography
can be found in the expository articles \cite{Galbraith2005,Paterson2005}, and in the extensive research literature.

The elliptic curves suitable for implementing pairing-based systems should have a small embedding degree with respect to a large prime-order subgroup. We call them \emph{pairing-friendly elliptic curves}. More precisely, a pairing-friendly elliptic curve over a finite field $\F_q$
contains a subgroup of large prime order $r$ such that for some $k$, $r|q^{k}-1$ and $r\nmid q^{i}-1$ for $0<i<k$, and
the parameters $r,q$ and $k$ should satisfy the following conditions:
\begin{itemize}
\item $r$ should be large enough so that the Discrete Logarithm Problem (DLP) in an order-$r$ subgroup of $E(\F_q)$ is infeasible.

\item $k$ should be sufficiently large so that DLP in $\F_{q^{k}}^{*}$ is intractable.

\item $k$ should be small enough so that arithmetic in $\F_{q^{k}}$ is feasible.
\end{itemize}
Here, $k$ is called the \emph{embedding degree} of $E$ with respect to $r$, and the ratio $\frac{\log q}{\log r}$ is called the
\emph{rho-value} of $E$ with respect to $r$. There is a specific definition for pairing-friendly elliptic curve in \cite[Definition 2.3]{Freeman2010}, that is, it should meet $r\ge\sqrt{q}$ and $k\le \log_{2}(r)/8$, where $\log_2$ is the binary logarithm.

These conditions make pairing-friendly curves rare, and they can not be constructed by random generation. This naturally produces two important problems:
\begin{itemize}
\item Finding efficient constructions of pairing-friendly curves.

\item Analyzing these constructions, including the frequency of curves constructed, their efficiency, security level, etc.
\end{itemize}

The earliest constructions of pairing-friendly curves involved supersingular curves. However, on the one hand due to \emph{MOV attack}
\cite{Menezes1993}, Frey-R{\"u}ck reduction \cite{Frey1994} and most recently \cite{Hayashi}, supersingular curves are widely believed to have some cryptographic weaknesses;
on the other hand, for supersingular curves the embedding degree $k$ has only 5 choices, i.e. $k\in\{1,2,3,4,6\}$.
Thus, it seems quite important to construct ordinary curves with the above properties.

After consecutive efforts of many researchers, many methods for constructing ordinary curves have been found.
 An exhaustive survey can be found in \cite{Freeman2010}, where the authors gave a coherent framework of all existing constructions. Unfortunately, none of these constructions has been rigorously analyzed. Even heuristic analysis is far from sufficiency except for
the so-called \emph{MNT curves} \cite{Miyaji2001}. For the heuristic analysis of MNT curves, see \cite{Shparlinski2006,Shparlinski2012}.
Most recently, a heuristic asymptotic formula for the number of
isogeny classes of pairing-friendly curves over prime fields was presented in \cite{Boxall2012}, and some heuristic arguments about \emph{Barreto-Naehrig family} \cite{Barreto2006} were also given therein.

It is widely accepted that the \emph{Cocks-Pinch method} \cite{Cocks-Pinch} is one of the most flexible algorithms for constructing pairing-friendly curves, such as with many curves possible, with arbitrary embedding degree, with prime-order subgroups of nearly arbitrary size, and so on.
We recall it in Section \ref{Cocks-Pinch}.

In addition, pairing-friendly fields were introduced by Koblitz and Menezes \cite{Koblitz2005} as an
efficient way to implement cryptographic bilinear pairings. They define a field $\F_{p^k}$ as being pairing-friendly if the prime characteristic $p\equiv 1$ (mod 12) and the embedding degree $k=2^{i}3^{j}, i>0$. If $j=0$, it only needs $p\equiv 1$ (mod 4). Definitely pairing-friendly curves over pairing-friendly fields are attractive.

\subsection{Main results}

In this paper, firstly we give two different kinds of heuristics to justify the same asymptotic formula about the Cocks-Pinch method. This confirms the general consensus that most
curves constructed by this method have rho-value around 2. One is borrowed directly from \cite{Boxall2012}, the other is based on the Bateman-Horn conjecture. Finally, we will see that the formula is compatible with numerical data. The reason we present the latter one is that in many cases the Bateman-Horn conjecture is indispensable for such heuristics, for example see \cite{Shparlinski2012}. Through the comparison here, we can say that such heuristics based on the Bateman-Horn conjecture are likely to be reasonable.

Secondly, we present the first known heuristics about pairing-friendly curves over pairing-friendly fields. The heuristics suggest that any efficient construction of pairing-friendly curves is also an efficient construction of such curves over pairing-friendly fields, naturally including the Cocks-Pinch method. Especially, the heuristics will be confirmed by the numerical data from the Cocks-Pinch method.

\subsection{Preliminary and Notation}
Let $\Phi_{k}$ be the $k$-th cyclotomic polynomial. The existing constructions of ordinary curves with small embedding degree typically work in the following two steps.
\begin{enumerate}
\item Find an odd prime $r$, integers $k\ge 2$ and $t$, and a prime power $q$ such that
\begin{equation}\label{condition1}
|t|\le 2\sqrt{q},\quad \gcd(q,t)=1, \quad r|q+1-t,\quad r|\Phi_{k}(q).
\end{equation}

\item Construct an elliptic curve $E$ over $\F_q$ with $|E(\F_q)|=q+1-t$.
\end{enumerate}

Since $r|\Phi_{k}(q)$, $k$ is the multiplicative order of $q$ modulo $r$ and then $k|r-1$.
In order to satisfy practical requirements, $k$ should be reasonably small, while the rho-value should be as small as possible, preferably close to 1.

Unfortunately, the second step above is feasible only if $t^{2}-4q$ has a very small square-free part; that is, if the so-called
\emph{CM equation}
\begin{equation}\label{condition2}
4q=t^{2}+Du^{2}
\end{equation}
with some integers $u$ and $D$, where $D$ is a small square-free positive integer. In this case, for example when $D\le 10 ^{13}$ (see \cite{Sutherland2011}), $E$ can be efficiently constructed via the \emph{CM method} (see \cite[Section 18.1]{Avanzi2005}). Here, $D$ is called the \emph{CM discriminant} of $E$.

For the imaginary quadratic field $\Q(\sqrt{-D})$, let $h_{D}$ be the class number of $\Q(\sqrt{-D})$ and $w_{D}$ the number of roots of unity in $\Q(\sqrt{-D})$. We denote its discriminant by $D^{*}$. Then put $e(k,D)=2$ if $D^{*}|k$ (namely $\Q(\sqrt{-D})\subseteq \Q(\zeta_k)$), otherwise put $e(k,D)=1$.

Recall that a well-known kind of constructions of pairing-friendly curves with $k$ and $D$ fixed is called the \emph{complete polynomial family}.
Briefly speaking, the idea is to parameterize $r,t,  q,u$ as polynomials and then choose $r(x),t(x),  q(x),u(x)$ satisfying Conditions (\ref{condition1}) and (\ref{condition2}) for any $x$. Here we define the ratio $\frac{\deg q(x)}{\deg r(x)}$ as the \emph{rho-value} of the family. See \cite[Section 2.1]{Freeman2010} for more details.

Throughout the paper, we use the Landau symbols $O$ and $o$ and the Vinogradov symbol $\ll$. We recall that the assertions $U=O(V)$ and $U\ll V$ are both equivalent to the inequality $|U|\le cV$ with some constant $c$, while $U=o(V)$ means that $U/V\to 0$.

In this paper, we also use the asymptotic notation $\sim$. Let $f$ and $g$ be two real functions with respect to $x$, both of them are strictly positive for sufficiently large $x$. We say that $f$ is asymptotically equivalent to $g$ as $x\to \infty$ if $f(x)/g(x)\to 1$ when $x\to \infty$, denoted by $f(x)\sim g(x)$.

\section{Heuristics of the Cocks-Pinch method}\label{Cocks-Pinch}

\subsection{Background on the Cocks-Pinch method}
In an unpublished manuscript \cite{Cocks-Pinch}, Cocks and Pinch proposed an algorithm for constructing pairing-friendly curves with arbitrary embedding degree. More precisely, see \cite[Theorem 4.1]{Freeman2010} or \cite[Algorithm IX.4]{Galbraith2005}, fix an embedding degree $k$ and a
CM discriminant $D$, then execute the following steps:

\begin{enumerate}[Step 1.]
\item Choose a prime $r$ such that $k|r-1$ and $-D$ is square modulo $r$.\label{step1}

\item Choose an integer $g$ which is a primitive $k$-th root of unity in $(\Z/r\Z)^{*}$. \label{step2}

\item Put $t^{\prime}=g+1$ and choose an integer $u^{\prime}\equiv (t^{\prime}-2)/\sqrt{-D}$ (mod $r$).\label{step3}

\item Let $t\in \Z$ be congruent to $t^{\prime}$ modulo $r$, and let $u\in \Z$ be congruent to $u^{\prime}$ modulo $r$. Put $q=(t^{2}+Du^{2})/4$.\label{step4}

\item If $q$ is an integer and prime, then there exists an elliptic curve $E$ over $\F_q$ with an order-$r$ subgroup and embedding degree $k$.
If $D$ is not to large, then $E$ can be efficiently constructed via the CM method.\label{step5}

\end{enumerate}

First, we notice that every triple $(r,t,q)$ satisfying Conditions (\ref{condition1}) and (\ref{condition2}) with $q$ prime can be generated by the
Cocks-Pinch method. 

Given a real number $\rho>0$, let $F_{k,D,\rho}(x)$ be the number of triples $(r,t,q)$ constructed by the Cocks-Pinch method with fixed $k$ and $D$ such that $q$ is an odd prime, $r\le x$ and $q\le r^{\rho}$. The previous paragraph implies that there is a natural one to one correspondence between the triples $(r,t,q)$ here and the triples in \cite[Estimate 1]{Boxall2012}. The reason we use the parameter $q$ in the triples here is that we want to underline its importance.

In the sequel, first we will extend \cite[Estimate 1]{Boxall2012} to all $\rho>1$ for $F_{k,D,\rho}(x)$, for the sake of completeness.
Then we will give another approach to this heuristic formula by applying the Bateman-Horn conjecture. In Section \ref{data}, we will see that this formula is compatible with numerical data.

\subsection{Heuristics from algebraic number theory}

As the above discussions, Boxall \cite[Estimate 1]{Boxall2012} actually got a heuristic asymptotic formula for $F_{k,D,\rho}(x)$ when $1<\rho<2$.

\begin{estimate}[\cite{Boxall2012}]\label{F_{k,D}}
Given an integer $k\ge 3$, a positive square-free integer $D$ and a real $\rho>1$.
 Suppose that
 \begin{enumerate}
\item $(k,D)\ne (3,3),(4,1)$ and $(6,3)$; \label{assump1}
\item If there exists a complete polynomial family $(r(x), t(x), q(x))$ of pairing-friendly curves with rho-value 1, embedding degree $k$ and CM discriminant $D$, then $\rho>1+\frac{1}{\deg r(x)}$. \label{assump2}
\end{enumerate}
Then we have the following heuristic asymptotic formula
\begin{equation}\label{main}
F_{k,D,\rho}(x)\sim \frac{e(k,D)w_{D}}{2\rho h_{D}}\int_{5}^{x}\frac{dz}{z^{2-\rho}(\log z)^2}.
\end{equation}
\end{estimate}
\begin{proof}
For the heuristic arguments of \cite[Estimate 1]{Boxall2012}, the condition $1<\rho<2$ is only used in \cite[Page 87, Step 3]{Boxall2012}.
Notice that the number of prime ideals of $\Q(\sqrt{-D})$ with norm bounded by $x$ is asymptotically equivalent to $\frac{x}{\log x}$ as $x\to \infty$, but the number of prime ideals with norm bounded by $x$ and underlying prime number inert is $O(\frac{\sqrt{x}}{\log \sqrt{x}})$. So, as $x\to\infty$, we can get the same heuristic formula when $\rho\ge 2$.

Since $k|r-1$ and $k\ge 3$, we have $r\ge k+1\ge 4$. So $r\ge 5$. Then we choose the integral interval $[5,x]$.

\end{proof}

As explained in \cite{Boxall2012}, without the two assumptions in Estimate \ref{F_{k,D}}, the asymptotic formula may not hold any more. In particular, if there exists a complete polynomial family  with rho-value 1, embedding degree $k$ and CM discriminant $D$, then this family can generate more triples than predicted by (\ref{main}). For example, the Barreto-Naehrig family is currently the only known complete polynomial family with rho-value 1, for this family $k=12$, $D=3$ and $\deg r(x)=4$, see Table \ref{ideal} for numerical data.

Now we want to say more about the parameters in (\ref{main}). It is well-known that $w_D$ is given by the following formula:
\begin{equation}
w_{D}=\left\{ \begin{array}{ll}
                 4 & \textrm{if $D=1$},\\
                 6 & \textrm{if $D=3$},\\
                 2 & \textrm{if $D=2$ or $D>3$}.
                 \end{array} \right.
\notag
\end{equation}
Furthermore, by the well-known Dirichlet's class number formula of imaginary quadratic fields (for example see \cite[Exercise 10.5.12]{Murty2004}), we know
\begin{equation}\label{class number}
h_{D}=\left\{ \begin{array}{ll}
                 \sqrt{D}w_{D}L_{D}/\pi & \textrm{if $D\equiv 1,2$ (mod 4)},\\
                 \sqrt{D}w_{D}L_{D}/(2\pi) & \textrm{if $D\equiv 3$ (mod 4)},
                 \end{array} \right.
\end{equation}
where $L_{D}=\sum\limits_{n=1}^{\infty}\left(\frac{D^{*}}{n}\right)/n=\prod\limits_{\textrm{prime $p$}}\left(1-\left(\frac{D^{*}}{p}\right)/p\right)^{-1}$, $D^{*}$ is the discriminant of $\Q(\sqrt{-D})$ and $(\frac{\cdot}{\cdot})$ is the Kronecker  symbol.

Based on the following lemma, we can get another version of the above proposition, that is,
\begin{equation}\label{F_k}
F_{k,D,\rho}(x)\sim \frac{e(k,D)w_{D}}{2\rho(\rho-1) h_{D}}\frac{x^{\rho-1}}{(\log x)^2},
\end{equation}
 see also \cite[Formula (0.1)]{Boxall2012}. We are sure that the lemma is well-known. It is more convenient to give a simple proof rather than find some references. We will use it later.
\begin{lemma}\label{integral}
For any real numbers $a,m, s$ with $a>1$ and $s<1$, we have
$$
\int_{a}^{x}\frac{dz}{z^s(\log z)^{m}}\sim \frac{x^{1-s}}{(1-s)(\log x)^{m}}.
$$
\end{lemma}
\begin{proof}
Integrating by parts, we obtain
$$
\int_{a}^{x}\frac{dz}{z^s(\log z)^{m}}=\frac{z^{1-s}}{(1-s)(\log z)^{m}}\Big\vert_{a}^{x}+\frac{m}{1-s}\int_{a}^{x}\frac{dz}{z^s(\log z)^{m+1}},
$$
and
$$
\int_{a}^{x}\frac{dz}{z^s(\log z)^{m+1}}=\frac{z^{1-s}}{(1-s)(\log z)^{m+1}}\Big\vert_{a}^{x}+\frac{m+1}{1-s}\int_{a}^{x}\frac{dz}{z^s(\log z)^{m+2}}.
$$
We choose a positive real number $A$ such that $A>a$ and $\log A>\frac{m+1}{1-s}$. Notice that for $x>A$, we have
$$
\int_{a}^{x}\frac{dz}{z^s(\log z)^{m+2}}\le \int_{a}^{A}\frac{dz}{z^s(\log z)^{m+2}}+\frac{1}{\log A}\int_{A}^{x}\frac{dz}{z^s(\log z)^{m+1}}.
$$
Then we get
$$
\int_{a}^{x}\frac{dz}{z^s(\log z)^{m+1}}\ll \frac{x^{1-s}}{(1-s)(\log x)^{m+1}}.
$$
Finally, we have
$$
\int_{a}^{x}\frac{dz}{z^s(\log z)^{m}}\sim \frac{x^{1-s}}{(1-s)(\log x)^{m}}.
$$
\end{proof}

It is widely accepted that the rho-value of curves produced by the Cocks-Pinch method tends to be around $2$.
From (\ref{F_k}) we can easily see that when $\rho$ is close to $1$, the curves with relevant rho-value
are rare among the whole family constructed by the Cocks-Pinch method.

\subsection{Heuristics from the Bateman-Horn conjecture}

The Bateman-Horn conjecture has been used to analyze some constructions of pairing-friendly elliptic curves, see \cite{Boxall2012,Shparlinski2012}.
In this subsection, applying the Bateman-Horn conjecture we will give another approach to justify the heuristic asymptotic formula of $F_{k,D,\rho}(x)$ in Estimate \ref{F_{k,D}}.

The Bateman-Horn conjecture provides a conjectured density for the positive integers at which a given system of polynomials all have prime values,
see \cite{Bateman1962}. We recall it here for the conveniences of readers.

Given any finite set $\mathcal{F}=\{f_1,f_2,\cdots,f_m\}$ consisting of irreducible polynomials $f_1(T),\cdots,f_m(T)\in\Z[T]$ with positive leading coefficients and such that there is no prime $p$ with $p|f_1(n)\cdots f_m(n)$ for every integer $n\ge 1$, the Bateman-Horn conjecture says
\begin{equation}
|\{1\le n\le X: f_1(n),\cdots, f_m(n)\textrm{ are all prime}\}|\sim \frac{C(\mathcal{F})}{\deg f_1\cdots\deg f_m}\int_{2}^{X}\frac{dz}{(\log z)^{m}},
\end{equation}
where $C(\mathcal{F})$ is given by the conditionally convergent infinite product
$$
C(\mathcal{F})=\prod\limits_{\textrm{$p$ prime}}\frac{1-\omega_{p}(\mathcal{F})/p}{(1-1/p)^{m}},
$$
and
$$
\omega_{p}(\mathcal{F})=|\{1\le n\le p: f_1(n)\cdots f_m(n)\equiv \textrm{0 (mod $p$)}\}|.
$$

Based on Lemma \ref{integral}, we can get another version of the Bateman-Horn conjecture, that is,
\begin{equation}
|\{1\le n\le X: f_1(n),\cdots, f_m(n)\textrm{ are all prime}\}|\sim \frac{C(\mathcal{F})}{\deg f_1\cdots\deg f_m}\frac{X}{(\log X)^{m}},
\end{equation}
which we will use in the sequel.

Notice that the ring of integer of $\Q(\sqrt{-D})$ is $\Z\oplus\Z\frac{1+\sqrt{-D}}{2}$ if $D\equiv$ 3 (mod 4), and it is
$\Z\oplus\Z\sqrt{-D}$ if $D\equiv$ 1 or 2 (mod 4). Since the element $\alpha=\frac{t+u\sqrt{-D}}{2}$ is an algebraic integer of $\Q(\sqrt{-D})$,
$t$ and $u$ must have the same parity if $D\equiv$ 3 (mod 4), and otherwise both of them must be even.

\begin{estimate}\label{F2}
For any integer $k\ge 3$, and positive square-free integer $D\equiv 1,2$ {\rm (mod 4)}, under the same assumptions as Estimate \ref{F_{k,D}},   we heuristically have
\begin{equation}
F_{k,D,\rho}(x)\sim  \frac{e(k,D)w_D}{2\rho h_{D}}\int_{5}^{x}\frac{dz}{z^{2-\rho}(\log z)^2}.
\end{equation}
\end{estimate}
\begin{proof}
We investigate the first four steps of the Cocks-Pinch method one by one.

Let $r\ge 2$ be any integer. The probability that $r$ is prime is $1/\log r$,
here we use the regular heuristic that the probability of a random integer $n$ to be prime is $1/\log n$. Since $k$ has finitely many prime factors,
for an arbitrary prime $r$, the probability that $r\nmid k$ is 1.
 Notice that there are $\varphi(k)$ residue classes modulo $k$ which consist of integers prime to $k$, the probability that $r$ is prime and $k |r-1$ is $\frac{1}{\varphi(k)\log r}$.

Since $k|r-1$, $r$ splits completely  over $\Q(\zeta_k)$.    Therefore, if $\Q(\sqrt{-D})\subseteq \Q(\zeta_k)$, i.e. the discriminant of $\Q(\sqrt{-D})$ divides $k$,  then $r$ splits completely over $\Q(\sqrt{-D})$. Thus, $-D$ is square modulo $r$. Otherwise, if $\Q(\sqrt{-D})\not\subseteq \Q(\zeta_k)$, the probability that $-D$ is square modulo $r$ is $1/2$. So the probability that $-D$ is square modulo $r$ is $e(k,D)/2$.

When $r$ is fixed, the number of choices of $g$ is $\varphi(k)$. After fixing $g$, $t^{\prime}$ is fixed and $u^{\prime}$ has two choices.

Thus, for an arbitrary integer $r\ge 2$, the probability that $r$ satisfies Steps \ref{step1}, \ref{step2} and \ref{step3} is $e(k,D)/\log r$. Moreover, it also needs that $r\ge 5$. In the sequel, we investigate Step \ref{step4}.

Since $D\equiv 1,2$ (mod 4), $t$ and $u$ must be even. So it is equivalent to count the number of integer pairs $(t,u)$ such that $q=t^2+Du^2$ is prime
with $q\le r^{\rho}$. Then for the integers $t$ and $u$, we have $|t|\le \sqrt{r^{\rho}}$ and $|u|\le \sqrt{r^{\rho}/D}$. Notice that the ratio between the area of the ellipse $\Lambda: t^2+Du^2=r^{\rho}$ and that of the rectangle $\Omega=\{(t,u): |t|\le \sqrt{r^{\rho}}, |u|\le \sqrt{r^{\rho}/D}\}$ is $\pi/4$. Now we assume that the ratio of the number of integer pairs $(t,u)$ in $\Lambda$ and that in $\Omega$ is $\pi/4$. Subsequently, we first count the number of $(t,q)$ with $q=t^2+Du^2$ prime, $t\le \sqrt{r^{\rho}}$ and $u\le \sqrt{r^{\rho}/D}$, and then to get the final result we need to multiply this amount by $\pi/4$ .

For every integer $u\ge 1$, let $f_u(T)=T^2+Du^2\in\Z[T]$. For $\mathcal{F}=\{f_u\}$, it satisfies the required conditions. By the Bateman-Horn conjecture, we have
\begin{equation}
|\{1\le t\le \sqrt{r^{\rho}}: f_u(t)\textrm{ is prime}\}|\sim \frac{C(f_u)\sqrt{r^{\rho}}}{\rho\log r},
\notag
\end{equation}
where
$$
C(f_u)=\prod\limits_{\textrm{$p$ prime}}\frac{1-\omega_{p}(f_u)/p}{1-1/p},
$$
and
$$
\omega_{p}(f_u)=|\{1\le n\le p: n^{2}\equiv -Du^{2} \textrm{ (mod $p$)}\}|.
$$
It is easy to see that
\begin{equation}
\omega_{p}(f_u)=\left\{ \begin{array}{ll}
                 1 & \textrm{if $p=2$ or $p|u$},\\
                 \left(\frac{-D}{p}\right)+1 & \textrm{if $p\ge 3$ and $p\nmid u$}.
                 \end{array} \right.
\notag
\end{equation}
Put
$$
g(u)=\prod\limits_{\textrm{$p\ge 3, p|u$}}\frac{p-1}{p-1-\left(\frac{-D}{p}\right)}.
$$
We also set $g(2^{n})=1$ for any integer $n\ge 0$. This makes $g(u)$ a multiplicative function. Notice that
$$
C(f_1)=C(f_2)=\prod\limits_{\textrm{$p\ge 3$}}\frac{p-1-\left(\frac{-D}{p}\right)}{p-1}.
$$
Obviously, $C(f_u)=C(f_1)\cdot g(u)$. Then we have
$$
\sum\limits_{1\le u\le \sqrt{r^{\rho}/D}}\frac{C(f_u)\sqrt{r^{\rho}}}{\rho\log r}=\frac{C(f_1)\sqrt{r^{\rho}}}{\rho\log r}\sum\limits_{1\le u\le \sqrt{r^{\rho}/D}}g(u).
$$

Here we need an asymptotic formula for
$$
S(X)=\sum\limits_{1\le u\le X}g(u).
$$
Notice that $g(u)$ is a multiplicative function and $1-1/p\le g(p)\le 1+\frac{3}{p}$ for any prime $p$. Recall the Mertens' second theorem
\begin{equation}
\sum\limits_{\textrm{$p\le X$}}\frac{1}{p}=\log\log X+B_1+o(1),
\notag
\end{equation}
where $B_1$ is an absolute constant (see \cite[Theorem 427]{Hardy}).
Then, we get
$$
\sum\limits_{\textrm{$p\le X$}}g(p)=\pi(X)+O(\log\log X),
$$
where $\pi(X)$ is the number of primes less than or equal to $X$. Then by \cite[Proposition 4]{Finch2010} , we have
$$
S(X)=(C_{g}+o(1))X,
$$
where
$$C_g=\prod\limits_{\textrm{$p$}}(1+\frac{g(p)}{p}+\frac{g(p^{2})}{p^{2}}+\cdots)(1-\frac{1}{p}).$$

Notice that $g(p^{n})=g(p)$ for any prime $p$ and any $n\ge 1$. Then, we have
$$
C_{g}=\prod\limits_{\textrm{$p\ge 3$}}\frac{p-1}{p}\left(1+\frac{1}{p-1-\left(\frac{-D}{p}\right)}\right).
$$
Thus,
$$
C(f_1)C_{g}=\prod\limits_{\textrm{$p\ge 3$}}\left(1-\left(\frac{-D}{p}\right)/p\right)=L_{D}^{-1},
$$
where $L_D$ has been defined in (\ref{class number}).
Hence,
\begin{align*}
\sum\limits_{1\le u\le \sqrt{r^{\rho}/D}}\frac{C(f_u)\sqrt{r^{\rho}}}{\rho\log r}&=(L_{D}^{-1}+o(1))\frac{r^{\rho}}{\sqrt{D}\rho\log r}\\
&\sim \frac{r^{\rho}}{\rho L_{D}\sqrt{D}\log r}=\frac{w_{D}r^{\rho}}{\pi\rho h_{D}\log r}.
\end{align*}

Note that $t$ can be a negative integer. We also note that if $t^{\prime}$ and $u^{\prime}$ are fixed, then the residue classes modulo $r$ which $t$ and $u$ belong to are also fixed. So the expected number of pairs $(t,q)$ associated to a triple $(r,t^{\prime},u^{\prime})$ with $q\le r^{\rho}$  is asymptotically equivalent to
\begin{equation}
\frac{\pi}{4}\cdot\frac{w_{D}r^{\rho}}{\pi\rho h_{D}\log r}\cdot 2 \cdot\frac{1}{r^{2}}=\frac{w_D}{2\rho h_{D}r^{2-\rho}\log r},
\notag
\end{equation}
as $r\to \infty$.

Therefore, we have
\begin{align*}
F_{k,D,\rho}(x)&\sim \sum\limits_{5\le r\le x}\frac{e(k,D)}{\log r}\cdot \frac{w_D}{2\rho h_{D}r^{2-\rho}\log r}\\
&\sim \frac{e(k,D)w_D}{2\rho h_{D}}\int_{5}^{x}\frac{dz}{z^{2-\rho}(\log z)^2}.
\notag
\end{align*}
\end{proof}

For the Cocks-Pinch method, it is fortunate that we can apply two different kinds of heuristics. But in general, the Bateman-Horn conjecture is indispensable when investigating the constructions of pairing-friendly curves. Estimate \ref{F2} tells us that such investigations based on the Bateman-Horn conjecture are likely to be reasonable.

\subsection{Remark}
Boneh, Rubin and Silverberg \cite{Boneh2011} have found that the Cocks-Pinch method can be used to construct elliptic curves with embedding degree $k$ with respect to $r$, where $r$ is a large composite number. This kind of elliptic curves was first used by Boneh, Goh and Nissim \cite{Boneh2005} for partial homomorphic encryption, and now they have a number of other important applications in cryptography. The methods of this section could also be applied to obtain heuristic estimates in this context.

\section{Involving Pairing-friendly Fields}\label{field}
In this section, we want to heuristically count the number of triples $(r,t,q)$ constructed by the Cocks-Pinch method such that $q$ is a prime and $q\equiv 1$ (mod 4 or 12) .

Let $G_{k,D,\rho}(x)$ be the number of triples $(r,t,q)$ constructed by the Cocks-Pinch method with fixed $k$ and $D$ such that $q$ is an odd prime, $q\equiv 1$ (mod 4), $r\le x$ and $q\le r^{\rho}$. Let $H_{k,D,\rho}(x)$ be the number of such triples $(r,t,q)$ when we furthermore require that $q\equiv 1$ (mod 12).

From the CM equation: $q=\frac{t^2+Du^2}{4}$, it is easy to see that $q\equiv 1$ (mod 12) if and only if $q\equiv 1$ (mod 4) and $t^2+Du^2\equiv 1$ (mod 3).

First, we study the probability that $t^2+Du^2\equiv 1$ (mod 3).

\begin{proposition}\label{03}
If $3|D$, then we always have $t^2+Du^2\equiv 1$ {\rm (mod 3)}.
\end{proposition}
\begin{proof}
Since $3|D$, $t^2+Du^2\equiv t^2 \equiv 1$ (mod 3) holds only if $3\nmid t$. Assume that $3|t$. Then we have $3|q$, thus $q=3$. Then $t=0,D=3$ and $u=\pm 2$. Since $r|q+1\pm t$ and $r\ge 5$, there is no possible $r$. So we must have $3\nmid t$, and thus we always have $t^2+Du^2\equiv 1$ (mod 3).
\end{proof}

\begin{corollary}\label{GH}
If $3|D$, then we always have $G_{k,D,\rho}(x)=H_{k,D,\rho}(x)$.
\end{corollary}

\begin{proposition}\label{14}
Assume that $k\ge 3$ and $D\equiv 1$ {\rm (mod 4)}. Then the following hold.

{\rm (1)} $G_{k,D,\rho}(x)= F_{k,D,\rho}(x)$.

{\rm(2)} If furthermore $D\equiv 0$ {\rm (mod 3)}, we have $H_{k,D,\rho}(x)= F_{k,D,\rho}(x)$.


\end{proposition}
\begin{proof}
(1)
Since $D\equiv 1$ (mod 4), for a constructed prime $q=\frac{t^2+Du^2}{4}$, $t$ and $u$ must be even. Notice that since $D$ and $q$ are odd, $\frac{t}{2}$ and $\frac{u}{2}$ must have different parities. Thus it is always true that $q\equiv 1$ (mod 4). So we prove (1).

(2) Since $q\equiv 1$ (mod 4), we know that $q\equiv 1$ (mod 12) if and only if $t^2+Du^2\equiv 1$ {\rm (mod 3)}. Then (2) follows from Proposition \ref{03}.
\end{proof}

\begin{proposition}\label{123}
No matter $D\equiv 1$ or $2$  {\rm (mod 3)}, the formula $t^2+Du^2\equiv 1$ {\rm (mod 3)} is true with the probability of $1/2$.
\end{proposition}
\begin{proof}
 Consider the element $\alpha=\frac{t+u\sqrt{-D}}{2}$ of the imaginary quadratic field $\Q(\sqrt{-D})$.
If we denote by $\NN(\cdot)$ the absolute norm of $\Q(\sqrt{-D})$, then $\NN(\alpha)=q$. Then, the event that $q$ is a prime is equivalent to that $q$ splits in $\Q(\sqrt{-D})$ as a product of two principal prime ideals. By the properties of Hilbert class fields, it is equivalent to that $q$ splits completely in $H_D$, where $H_D$ is the Hilbert class field of $\Q(\sqrt{-D})$.

On the other hand, since $q=\frac{t^2+Du^2}{4}$, we see that $t^2+Du^2\equiv 1$ {\rm (mod 3)} if and only if $q\equiv 1$ {\rm (mod 3)}. Furthermore, whether $q\equiv 1$ {\rm (mod 3)} or not only depends on the splitting behavior of $q$ in $\Q(\sqrt{-3})$.

Notice that since $3\nmid D$, we have $H_D\cap \Q(\sqrt{-3})=\Q$. Therefore, the splitting behaviors of $q$ in $H_D$ and in $\Q(\sqrt{-3})$ are independent. Subsequently, the two events that $q$ is a prime and $q\equiv 1$ {\rm (mod 3)} are independent. So by Dirichlet's theorem on arithmetic progressions, asymptotically as $X\to \infty$ half of the primes $q\le X$ constructed by the Cocks-Pinch method satisfy $q\equiv 1$ {\rm (mod 3)}. Then, the desired result follows.

\end{proof}

\begin{proposition}
Assume that $k\ge 3$, $D\equiv 1$ {\rm (mod 4)} and $D\equiv 1,2$ {\rm (mod 3)}. we have $H_{k,D,\rho}(x)\sim \frac{1}{2}F_{k,D,\rho}(x)$.
\end{proposition}
\begin{proof}
Since $D\equiv 1$ {\rm (mod 4)}, we have $q\equiv 1$ (mod 4). So, $q\equiv 1$ (mod 12) if and only if $t^2+Du^2\equiv 1$ {\rm (mod 3)}. Then the desired result follows from Proposition \ref{123}.
\end{proof}

For the case $D\equiv 2,3$ (mod 4), the heuristics are also straightforward.

\begin{estimate}\label{234}
Assume that $k\ge 3$ and $D\equiv 2,3$ (mod 4). Then the following hold heuristically.

{\rm(1)} $G_{k,D,\rho}(x)\sim \frac{1}{2}F_{k,D,\rho}(x)$.

{\rm(2)} If furthermore $D\equiv 0$ {\rm (mod 3)}, we have $H_{k,D,\rho}(x)\sim \frac{1}{2}F_{k,D,\rho}(x).$

{\rm(3)} If furthermore $D\equiv 1,2$ {\rm (mod 3)}, we have $H_{k,D,\rho}(x)\sim \frac{1}{4}F_{k,D,\rho}(x).$


\end{estimate}
\begin{proof}
We divide the proof into three parts according to three cases.

(I) Assume that $D\equiv 2$ (mod 4).

(1)
Since $D\equiv 2$ (mod 4), for a constructed prime $q=\frac{t^2+Du^2}{4}$, $t$ and $u$ must be even.
Notice that since $D$ is even and $q$ is odd, $\frac{t}{2}$ must be odd. Then $(\frac{t}{2})^2+D(\frac{u}{2})^2\equiv 1$ (mod 4) holds only if $\frac{u}{2}$ is even. Suppose that the even parity and odd parity of $\frac{u}{2}$ have the same probability. Then the probability that $q\equiv 1$ (mod 4) is $1/2$, which proves (1).

(2) and (3)
By Propositions \ref{03} and \ref{123}, the probability that $t^2+Du^2\equiv 1$ (mod 3) is $1,1/2$ or $1/2$ corresponding to $D\equiv 0,1$ or $2$  (mod 3), respectively. Since the two events $q\equiv 1$ (mod 4) and $t^2+Du^2\equiv 1$ (mod 3) are independent. Then we can get the desired results.

(II) Assume that $D\equiv 7,15$ (mod 16).

(1)
Since $D\equiv 3$ (mod 4), for a constructed prime $q=\frac{t^2+Du^2}{4}$, $t$ and $u$ must have the same parity. Furthermore, since $D\equiv 7,15$ (mod 16), we claim that $t$ and $u$ must be even.

Suppose that $t$ and $u$ are odd. Consider the CM equation $4q=t^2+Du^2$. Since $q$ is odd, $4q$ is equal to 4 or 12 modulo 16. But $t^2+Du^2$ is equal to 0 or 8 modulo 16 under the condition $D\equiv 7,15$ (mod 16). This leads to a contradiction.

Since $D$ and $q$ are odd, $\frac{t}{2}$ and $\frac{u}{2}$ must have different parities, which is naturally divided into two cases. Suppose that these two cases have the same probability. Then
the probability that $(\frac{t}{2})^2+D(\frac{u}{2})^2\equiv 1$ (mod 4) is $1/2$, which proves (1).

(2) and (3) Apply the same arguments as (I).

(III) Assume that $D\equiv 3,11$ (mod 16).

(1)
Since $D\equiv 3$ (mod 4), for a constructed prime $q=\frac{t^2+Du^2}{4}$, $t$ and $u$ must have the same parity. Furthermore, the two parities may  occur due to $D\equiv 3,11$ (mod 16).

First suppose that both of $t$ and $u$ are even. The deduction and the result of this case are the same as (II).

Now suppose that both of $t$ and $u$ are odd.
Notice that when $n$ is an odd integer, then $n^2\equiv 1, 9$ (mod 16). In this case, pairs $(t^2,u^2)$ can be divided into four classes according to the residue classes modulo 16 which $t^2$ and $u^2$ belong to. Suppose that all the four classes have the same probability. Then, when $D\equiv 3,11$ (mod 16),
the probability that $t^2+Du^2\equiv 4$ (mod 16) is 1/2.

Notice that we obtain the same result for the two parities, then the probability that $q\equiv 1$ (mod 4) is $1/2$. So we prove (1).

(2) and (3) Apply the same arguments as (I).

\end{proof}

From the above results, the heuristics suggest that pairing-friendly curves over pairing-friendly fields can be efficiently constructed by the Cocks-Pinch method. Notice that there are 18 cases in the above proofs according to $D$ modulo 4 or 16 and $D$ modulo 3. In the next section, we will see that the heuristic results of this section are compatible with numerical data.

\begin{remark}
Notice that the above heuristics are independent of the Cocks-Pinch method, they can be applied to any other constructions. So we can say that any efficient construction of pairing-friendly curves is also an efficient construction of pairing-friendly curves over pairing-friendly fields.
\end{remark}

\section{Numerical Evidence}\label{data}
For testing Estimate \ref{F_{k,D}} and the heuristic results in Section \ref{field}, we write a programme in PARI/GP \cite{Pari} to execute the Cocks-Pinch method for searching all the triples $(r,t,q)$ with $k,D$ and $\rho$ being given, and $r$ in some interval $[a,b]$.

For given $k,D,\rho,a$ and $b$,  we denote by $N_{1}(k,D,\rho,a,b)$ the number of triples $(r,t,q)$ as in Estimate \ref{F_{k,D}} with $a\le r\le b$.
We denote by $N_{2}(k,D,\rho,a,b)$
(resp. $N_{3}(k,D,\rho,a,b)$) the number of such triples with $q\equiv 1$ (mod 4) (resp. $q\equiv 1$ (mod 12)). The outputs of the programme are these three quantities.

For $N_{1}(k,D,\rho,a,b)$, under some assumptions, there exists a heuristic formula from Estimate \ref{F_{k,D}}, stated as follows
\begin{equation}\label{Iab}
I(k,D,\rho,a,b)=\frac{e(k,D)w_{D}}{2\rho h_{D}}\int_{a}^{b}\frac{dz}{z^{2-\rho}(\log z)^2}.
\end{equation}
Let $I_1=e(k,D)^{-1}I(k,D,\rho,a,b)$. Then $I_1$ depends only on $D$ and $\rho$ but not on $k$.

In Section \ref{field}, we present some definite or heuristic results about the relations among $N_{i}(k,D,\rho,a,b)$, $i=1,2,3$. We list them as follows,
\begin{equation}\label{N12}
\left\{ \begin{array}{ll}
                 N_{2}(k,D,\rho,a,b)=N_{1}(k,D,\rho,a,b) & \textrm{if $D\equiv 1$ (mod 4)},\\
                 N_{2}(k,D,\rho,a,b)\approx\frac{1}{2}N_{1}(k,D,\rho,a,b) & \textrm{if $D\equiv 2,3$ (mod 4)};
                 \end{array} \right.
\end{equation}
\begin{equation}\label{N13}
\left\{ \begin{array}{ll}
                 N_{3}(k,D,\rho,a,b)=N_{1}(k,D,\rho,a,b) & \textrm{if $D\equiv 1$ (mod 4) and $D\equiv 0$ (mod 3)},\\
                 N_{3}(k,D,\rho,a,b)\approx\frac{1}{2}N_{1}(k,D,\rho,a,b) & \textrm{if $D\equiv 1$  (mod 4) and $D\equiv 1,2$ (mod 3)},\\
                 N_{3}(k,D,\rho,a,b)\approx\frac{1}{2} N_{1}(k,D,\rho,a,b) & \textrm{if $D\equiv 2,3$ (mod 4) and $D\equiv 0$ (mod 3)},\\
                 N_{3}(k,D,\rho,a,b)\approx\frac{1}{4}N_{1}(k,D,\rho,a,b) & \textrm{if $D\equiv 2,3$  (mod 4) and $D\equiv 1,2$ (mod 3)};
                 \end{array} \right.
\end{equation}
\begin{equation}\label{N23}
N_{2}(k,D,\rho,a,b)=N_{3}(k,D,\rho,a,b),\quad \textrm{if $D\equiv 0$ (mod 3)}.
\end{equation}

Similar as $I_1$ and by (\ref{N12}) and (\ref{N13}), we define $I_2$ and $I_3$ by analogy with $N_{2}(k,D,\rho,a,b)$ and $N_{3}(k,D,\rho,a,b)$,  respectively.

In this section, we will test all these results by numerical data.

In fact, \cite[Table 1 and Table 2]{Boxall2012} gave the values of $N_1(k,D,1.7,10^6,85 ~698 ~768)$ and $N_1(k,D,1.5,10^6,2\times 10^8)$ respectively, for $3\le k\le 30$ and all square-free integer $D$ with $D\le 15$. These two tables are compatible with (\ref{Iab}). In the sequel, we will choose more narrow interval $[a,b]$ and even choose $a=5$ for testing.

Here, for each entry in the following tables, if its actual value is not an integer, then it is rounded to the nearest
whole number.

Table \ref{1.8-1} gives the values of  $N_{1}(k,D,1.8,5,5\times 10^5)$ for all $k$ with $3\le k\le 18$ and various square-free $D$. Notice that in Section \ref{field} there are 18 cases according to $D$ modulo 4 (or 16) and $D$ modulo 3. The choices of $D$ here exactly cover all these cases.
The second line gives the value of $I_1$. The main part of the table contains the values of $N_{1}(k,D,1.8,5,5\times 10^5)$, the entries corresponding to values of $(k,D)$ with $e(k,D)=2$ are highlighted in bold; (\ref{Iab}) predicts that they should be close to $2I_1$ and thus roughly twice as large as the other entries in the same column. The entries corresponding to values of $(k,D)=(3,3),(4,1)$ and $(6,3)$ are left blank. The last line gives the average value of each column as $k$ varies from 3 to 18, the cases where $e(k,D)=2$ being counted with weight $\frac{1}{2}$ and the excluded values $(k,D)=(3,3),(4,1)$ and $(6,3)$ omitted. (\ref{Iab}) predicts that each of these averages should be close to $I_1$.

Table \ref{1.8-2} gives the values of  $N_{2}(k,D,1.8,5,5\times 10^5)$ for the same values of $(k,D)$ as Table \ref{1.8-1}. When $D\equiv 1$ (mod 4), (\ref{N12}) tells us that $N_{2}(k,D,1.8,5,5\times 10^5)=N_{1}(k,D,1.8,5,5\times 10^5)$ for each value of $(k,D)$. Otherwise, when $D\equiv 2,3$ (mod 4), (\ref{N12}) predicts that $N_{2}(k,D,1.8,5,5\times 10^5)$ should be close to half of $N_{1}(k,D,1.8,5,5\times 10^5)$.

Table \ref{1.8-3} gives the values of  $N_{3}(k,D,1.8,5,5\times 10^5)$ for the same values of $(k,D)$ as Table \ref{1.8-1}. (\ref{N13}) presents some definite or heuristic results about the relation between $N_{3}(k,D,1.8,5,5\times 10^5)$ and $N_{1}(k,D,1.8,5,5\times 10^5)$. For example, when $D\equiv 1$ (mod 4) and $D\equiv 0$ (mod 3), we have $N_{3}(k,D,1.8,5,5\times 10^5)=N_{1}(k,D,1.8,5,5\times 10^5)$. If $3|D$, (\ref{N23}) says that $N_{2}(k,D,1.8,5,5\times 10^5)=N_{3}(k,D,1.8,5,5\times 10^5)$.

The explanations of Tables \ref{2-1}, \ref{2-2} and \ref{2-3} are the same as Tables \ref{1.8-1}, \ref{1.8-2} and \ref{1.8-3}, respectively. Here, we choose another choices of $D$ to exactly cover the 18 cases in Section \ref{field}.

Although Tables \ref{1.8-1}--\ref{2-3} show that (\ref{N12})--(\ref{N23}) are supported by numerical data, there is some discrepancy between the expected values and the calculated values. For Tables \ref{1.8-1} and \ref{2-1}, this is expected. Because for the Bateman-Horn conjecture, there seems to be no good conjecture for the remainder, for example see \cite{Korevaar2010} for a discussion of the case of prime pairs. Thus, it may be also a hard problem to find one in the context of Estimate \ref{F_{k,D}}. The discrepancy in Tables \ref{1.8-2}, \ref{1.8-3}, \ref{2-2} and \ref{2-3} arises from the assumptions made in Section \ref{field}, it seems also hard to make them more precisely. But most of the calculated values and all the average values are close to the expected values, this make us have confidence in the heuristic results.

Table \ref{ideal} gives the values of $N_{i}(12,3,\rho,10^4,10^8)$ for various $\rho$ and $i=1,2,3$. It shows that there is a big gap between $I(12,3,\rho,10^4,10^8)$ and $N_{1}(12,3,\rho,10^4,10^8)$ when $\rho<1.25$, because in this case the Barreto-Naehrig family makes the assumptions in Estimate \ref{F_{k,D}} not satisfied.  But in this exceptional case, (\ref{N12})--(\ref{N23}) are also compatible with numerical data.

\begin{table}
{\footnotesize
\caption{Values of $N_{1}(k,D,1.8,5,5\times 10^5)$ for various $k$ and $D$ (see Section \ref{data} for explanations) }\label{1.8-1}
\begin{tabular}{|p{0.8cm}|p{0.4cm}|p{0.4cm}|p{0.55cm}|p{0.35cm}|p{0.35cm}|p{0.4cm}|p{0.35cm}|p{0.4cm}|p{0.4cm}|p{0.35cm}|p{0.25cm}|p{0.25cm}|p{0.25cm}|p{0.35cm}|p{0.25cm}|p{0.35cm}|p{0.25cm}|p{0.35cm}|}
\hline
\quad $D$ & 1 & 2 & 3 & 5 & 6 & 7 & 10 & 11 & 15 & 19 & 21 & 23 & 31 & 35 & 39 & 43 & 47 & 123   \\ \hline

\quad $I_1$ & 377 & 189 & 566 & 94 & 94 & 189 & 94 & 189 & 94 & 189 & 47 & 63 & 63 & 94 & 47 & 189 & 38 & 94   \\ \hline


$k=3$  & 403 & 184 &  & 101 & 89 & 174 & 85 & 196 & 88 & 222 & 44 & 75 & 62 & 105 & 43 & 198 & 42 & 94  \\ \hline

\quad$4$&  & 174 & 583 & 112 & 107 & 221 & 97 & 211 & 87 & 196 & 58 & 49 & 68 & 101 & 49 & 203 & 32 & 126 \\ \hline

\quad$5$ & 429 & 217 & 570 & 105 & 96 & 218 & 101 & 184 & 92 & 213 & 48 & 60 & 63 & 100 & 53 & 212 & 37 & 103 \\ \hline

\quad$6$ & 388 & 193 &  & 95 & 105 & 199 & 109 & 180 & 88 & 182 & 52 & 57 & 62 & 107 & 60 & 206 & 44 & 116   \\ \hline

\quad$7$ & 420 & 193 & 627 & 96 & 92 & \textbf{374} & 94 & 195 & 104 & 202 & 42 & 75 & 74 & 88 & 44 & 218 & 34 & 109 \\ \hline

\quad$8$ & \textbf{802} & \textbf{365} & 592 & 130 & 85 & 172 & 88 & 200 & 103 & 200 & 57 & 71 & 54 & 89 & 51 & 176 & 44 & 111 \\ \hline

\quad$9$ & 371 & 182 & \textbf{1190} & 93 & 117 & 188 & 105 & 215 & 92 & 194 & 53 & 74 & 64 & 100 & 40 & 183 & 38 & 99   \\ \hline

\quad$10$ & 409 & 189 & 592 & 107 & 95 & 206 & 92 & 197 & 109 & 199 & 46 & 65 & 55 & 83 & 33 & 231 & 32 & 94 \\ \hline

\quad$11$ & 371 & 179 & 589 & 95 & 91 & 178 & 105 & \textbf{395} & 86 & 186 & 53 & 60 & 59 & 98 & 43 & 182 & 41 & 94 \\ \hline

\quad$12$ & \textbf{846} & 182 & \textbf{1230} & 85 & 87 & 206 & 101 & 181 & 85 & 189 & 50 & 57 & 69 & 91 & 49 & 197 & 28 & 96   \\ \hline

\quad$13$ & 380 & 197 & 622 & 99 & 79 & 180 & 102 & 200 & 89 & 206 & 47 & 60 & 61 & 93 & 40 & 172 & 35 & 106 \\ \hline

\quad$14$ & 413 & 190 & 582 & 78 & 83 & \textbf{423} & 99 & 197 & 89 & 202 & 55 & 68 & 57 & 94 & 49 & 217 & 29 & 97 \\ \hline

\quad$15$ & 405 & 184 & \textbf{1167} & 93 & 109 & 187 & 89 & 185 & \textbf{173} & 208 & 44 & 54 & 74 & 100 & 50 & 178 & 51 & 106 \\ \hline

\quad$16$ & \textbf{800} & \textbf{386} & 609 & 101 & 95 & 175 & 84 & 201 & 84 & 201 & 48 & 55 & 74 & 81 & 43 & 201 & 52 & 96 \\ \hline

\quad$17$ & 358 & 202 & 579 & 98 & 103 & 193 & 103 & 202 & 100 & 227 & 49 & 72 & 69 & 88 & 40 & 208 & 52 & 114   \\ \hline

\quad$18$ & 397 & 201 & \textbf{1203} & 87 & 91 & 195 & 100 & 209 & 90 & 195 & 54 & 55 & 79 & 106 & 51 & 190 & 43 & 91   \\ \hline

\ \ Avg & 398 & 190 & 596 & 98 & 95 & 193 & 97 & 197 & 97 & 201 & 50 & 63 & 65 & 95 & 46 & 198 & 40 & 103   \\ \hline
\end{tabular}
}
\end{table}

\begin{table}
{\footnotesize
\caption{Values of $N_{2}(k,D,1.8,5,5\times 10^5)$ for various $k$ and $D$  (see Section \ref{data} for explanations) }\label{1.8-2}
\begin{tabular}{|p{0.8cm}|p{0.4cm}|p{0.4cm}|p{0.55cm}|p{0.35cm}|p{0.35cm}|p{0.4cm}|p{0.35cm}|p{0.4cm}|p{0.4cm}|p{0.35cm}|p{0.25cm}|p{0.25cm}|p{0.25cm}|p{0.35cm}|p{0.25cm}|p{0.35cm}|p{0.25cm}|p{0.35cm}|}
\hline
\quad $D$ & 1 & 2 & 3 & 5 & 6 & 7 & 10 & 11 & 15 & 19 & 21 & 23 & 31 & 35 & 39 & 43 & 47 & 123   \\ \hline

\quad $I_2$ & 377 & 94 & 283 & 94 & 47 & 94 & 47 & 94 & 47 & 94 & 47 & 31 & 31 & 47 & 24 & 94 & 19 & 47   \\ \hline

$k=3$  & 403 & 84 &  & 101 & 52 & 84 & 34 & 101 & 48 & 109 & 44 & 38 & 28 & 46 & 25 & 93 & 18 & 41  \\ \hline

\quad$4$&  & 83 & 305 & 112 & 50 & 96 & 38 & 111 & 43 & 99 & 58 & 22 & 27 & 51 & 26 & 105 & 15 & 59 \\ \hline

\quad$5$ & 429 & 118 & 290 & 105 & 42 & 107 & 55 & 95 & 43 & 97 & 48 & 31 & 33 & 48 & 22 & 86 & 19 & 57 \\ \hline

\quad$6$ & 388 & 104 &  & 95 & 62 & 103 & 48 & 89 & 45 & 97 & 52 & 28 & 35 & 50 & 26 & 94 & 20 & 64  \\ \hline

\quad$7$ & 420 & 95 & 304 & 96 & 49 & \textbf{203} & 40 & 94 & 49 & 96 & 42 & 34 & 29 & 47 & 23 & 97 & 12 & 56 \\ \hline

\quad$8$ & \textbf{802} & \textbf{186} & 297 & 130 & 42 & 87 & 40 & 84 & 57 & 101 & 57 & 33 & 27 & 52 & 30 & 83 & 17 & 57 \\ \hline

\quad$9$ & 371 & 86 & \textbf{603} & 93 & 60 & 90 & 47 & 109 & 54 & 109 & 53 & 38 & 32 & 41 & 23 & 100 & 15 & 59   \\ \hline

\quad$10$ & 409 & 105 & 289 & 107 & 45 & 103 & 45 & 103 & 49 & 96 & 46 & 34 & 24 & 48 & 19 & 120 & 20 & 50 \\ \hline

\quad$11$ & 371 & 99 & 260 & 95 & 44 & 89 & 47 & \textbf{184} & 43 & 102 & 53 & 31 & 31 & 53 & 21 & 92 & 24 & 41 \\ \hline

\quad$12$ & \textbf{846} & 91 & \textbf{623} & 85 & 36 & 81 & 56 & 90 & 39 & 109 & 50 & 30 & 37 & 48 & 24 & 96 & 14 & 53   \\ \hline

\quad$13$ & 380 & 100 & 312 &99 & 32 & 96 & 49 & 110 & 56 & 102 & 47 & 30 & 23 & 46 & 17 & 80 & 11 & 54 \\ \hline

\quad$14$ & 413 & 92 & 271 & 78 & 47 & \textbf{215} & 52 & 104 & 49 & 110 & 55 & 38 & 35 & 42 & 26 & 118 & 13 & 41 \\ \hline

\quad$15$ & 405 & 93 & \textbf{574} & 93 & 61 & 103 & 41 & 93 & \textbf{86} & 112 & 44 & 30 & 32 & 49 & 16 & 93 & 22 & 48 \\ \hline

\quad$16$ & \textbf{800} & \textbf{195} & 314 & 101 & 43 & 89 & 38 & 111 & 44 & 102 & 48 & 25 & 38 & 46 & 25 & 109 & 26 & 46 \\ \hline

\quad$17$ & 358 & 96 & 296 & 98 & 55 & 93 & 50 & 94 & 49 & 113 & 49 & 33 & 34 & 40 & 26 & 112 & 28 & 55   \\ \hline

\quad$18$ & 397 & 105 & \textbf{653} & 87 & 47 & 101 & 51 & 96 & 51 & 102 & 54 & 28 & 45 & 53 & 24 & 97 & 18 & 34   \\ \hline

\ \ Avg & 398 & 96 & 297 & 98 & 48 & 96 & 46 & 99 & 48 & 104 & 50 & 31 & 32 & 48 & 23 & 98 & 18 & 51   \\ \hline
\end{tabular}
}
\end{table}

\begin{table}
{\footnotesize
\caption{Values of $N_{3}(k,D,1.8,5,5\times 10^5)$ for various $k$ and $D$ (see Section \ref{data} for explanations) }\label{1.8-3}
\begin{tabular}{|p{0.8cm}|p{0.4cm}|p{0.4cm}|p{0.55cm}|p{0.35cm}|p{0.35cm}|p{0.4cm}|p{0.35cm}|p{0.4cm}|p{0.4cm}|p{0.35cm}|p{0.25cm}|p{0.25cm}|p{0.25cm}|p{0.35cm}|p{0.25cm}|p{0.35cm}|p{0.25cm}|p{0.35cm}|}
\hline
\quad $D$ & 1 & 2 & 3 & 5 & 6 & 7 & 10 & 11 & 15 & 19 & 21 & 23 & 31 & 35 & 39 & 43 & 47 & 123   \\ \hline

\quad $I_3$ & 189 & 47 & 283 & 47 & 47 & 47 & 24 & 47 & 47 & 47 & 47 & 16 & 16 & 24 & 24 & 47 & 9 & 47   \\ \hline

$k=3$  & 193 & 42 &  & 46 & 52 & 43 & 14 & 53 & 48 & 59 & 44 & 20 & 9 & 25 & 25 & 45 & 8 & 41   \\ \hline

\quad$4$&  & 35 & 305 & 54 & 50 & 48 & 17 & 55 & 43 & 47 & 58 & 9 & 16 & 27 & 26 & 59 & 8 & 59 \\ \hline

\quad$5$ & 233 & 69 & 290 & 46 & 42 & 51 & 24 & 43 & 43 & 40 & 48 & 11 & 17 & 25 & 22 & 40 & 8 & 57 \\ \hline

\quad$6$ & 193 & 45 &  & 50 & 62 & 42 & 20 & 42 & 45 & 48 & 52 & 8 & 16 & 29 & 26 & 45 & 10 & 64  \\ \hline

\quad$7$ & 215 & 51 & 304 & 55 & 49 & \textbf{111} & 19 & 43 & 49 & 50 & 42 & 20 & 13 & 19 & 23 & 49 & 6 & 56 \\ \hline

\quad$8$ & \textbf{402} & \textbf{84} & 297 & 60 & 42 & 40 & 21 & 40 & 57 & 59 & 57 & 13 & 12 & 26 & 30 & 45 & 6 & 57 \\ \hline

\quad$9$ & 186 & 40 & \textbf{603} & 43 & 60 & 46 & 25 & 54 & 54 & 56 & 53 & 18 & 17 & 18 & 23 & 41 & 6 & 59   \\ \hline

\quad$10$ & 198 & 55 & 289 & 56 & 45 & 55 & 18 & 47 & 49 & 45 & 46 & 19 & 6 & 18 & 19 & 63 & 10 & 50 \\ \hline

\quad$11$ & 187 & 42 & 260 & 50 & 44 & 46 & 25 & \textbf{90} & 43 & 61 & 53 & 18 & 12 & 21 & 21 & 49 & 11 & 41 \\ \hline

\quad$12$ & \textbf{414 }& 37 & \textbf{623} & 44 & 36 & 43 & 28 & 52 & 39 & 55 & 50 & 21 & 21 & 25 & 24 & 46 & 6 & 53   \\ \hline

\quad$13$ & 203 & 53 & 312 & 42 & 32 & 37 & 24 & 59 & 56 & 47 & 47 & 13 & 10 & 23 & 17 & 31 & 3 & 54 \\ \hline

\quad$14$ & 209 & 50 & 271 & 42 & 47 & \textbf{104} & 27 & 50 & 49 & 53 & 55 & 17 & 17 & 22 & 26 & 66 & 6 & 41 \\ \hline

\quad$15$ & 185 & 57 & \textbf{574} & 46 & 61 & 49 & 15 & 49 & \textbf{86} &  45 & 44 & 16 & 18 & 34 & 16 & 50 & 10 & 48 \\ \hline

\quad$16$ & \textbf{401} & \textbf{106} & 314 & 45 & 43 & 43 & 13 & 54 & 44 & 45 & 48 & 12 & 13 & 24 & 25 & 64 & 14 & 46 \\ \hline

\quad$17$ & 179 & 45 & 296 & 52 & 55 & 41 & 23 & 41 & 49 & 58 & 49 & 18 & 20 & 22 & 26 & 57 & 14 & 55   \\ \hline

\quad$18$ & 199 & 49 & \textbf{653} & 46 & 47 & 45 & 23 & 49 & 51 & 54 & 54 & 13 & 24 & 26 & 24 & 42 & 3 & 34   \\ \hline

\ \ Avg & 199 & 48 & 297 & 49 & 48 & 46 & 21 & 49 & 48 & 51 & 50 & 15 & 15 & 24 & 23 & 50 & 8 & 51   \\ \hline

\end{tabular}
}
\end{table}

\begin{table}
{\footnotesize
\caption{Values of $N_{1}(k,D,2,5,10^5)$ for various $k$ and $D$ (see Section \ref{data} for explanations) }\label{2-1}
\begin{tabular}{|p{0.8cm}|p{0.35cm}|p{0.35cm}|p{0.35cm}|p{0.35cm}|p{0.35cm}|p{0.35cm}|p{0.35cm}|p{0.35cm}|p{0.35cm}|p{0.35cm}|p{0.35cm}|p{0.35cm}|p{0.35cm}|p{0.35cm}|p{0.35cm}|p{0.3cm}|p{0.35cm}|p{0.35cm}|}
\hline
\quad $D$ & 13 & 14 & 17 & 22 & 30 & 33 & 51 & 55 & 59 & 67 & 71 & 79 & 83 & 87 & 91 & 95 & 111 & 219   \\ \hline

\quad $I_1$ & 236 & 118 & 118 & 236 & 118 & 118 & 236 & 118 & 157 & 472 & 67 & 94 & 157 & 79 & 236 & 59 & 59 & 118   \\ \hline


$k=3$ & 248 & 115 & 132 & 240 & 109 & 131 & 256 & 135 & 156 & 513 & 89 & 91 & 149 & 81 & 229 & 56 & 58 & 117     \\ \hline

\quad$4$ & 251 & 118 & 119 & 250 & 138 & 116 & 227 & 128 & 194 & 498 & 77 & 106 & 167 & 86 & 242 & 75 & 67 & 144     \\ \hline

\quad$5$ & 249 & 117 & 126 & 272 & 100 & 109 & 227 & 119 & 170 & 488 & 66 & 92 & 149 & 78 & 250 & 57 & 63 & 144   \\ \hline

\quad$6$ & 261 & 118 & 104 & 273 & 133 & 106 & 229 & 118 & 171 & 514 & 72 & 85 & 203 & 77 & 249 & 62 & 64 & 107     \\ \hline

\quad$7$ & 244 & 131 & 130 & 229 & 122 & 132 & 250 & 120 & 152 & 498 & 79 & 104 & 180 & 81 & 240 & 64 & 65 & 133     \\ \hline

\quad$8$ & 277 & 111 & 128 & 238 & 111 & 116 & 269 & 124 & 127 & 480 & 79 & 93 & 150 & 72 & 238 & 65 & 54 & 112     \\ \hline

\quad$9$ & 264 & 139 & 136 & 248 & 118 & 109 & 236 & 125 & 164 & 522 & 62 & 104 & 156 & 75 & 256 & 56 & 74 & 109     \\ \hline

\quad$10$ & 233 & 126 & 125 & 246 & 131 & 103 & 230 & 102 & 168 & 486 & 58 & 103 & 161 & 78 & 254 & 66 & 54 & 121    \\ \hline

\quad$11$ & 240 & 117 & 126 & 223 & 131 & 135 & 239 & 124 & 156 & 441 & 65 & 101 & 174 & 96 & 253 & 59 & 58 & 99    \\ \hline

\quad$12$ & 243 & 125 & 110 & 245 & 116 & 128 & 211 & 125 & 151 & 503 & 75 & 87 & 152 & 79 & 244 & 63 & 52 & 126   \\ \hline

\quad$13$ & 256 & 124 & 121 & 237 & 118 & 116 & 285 & 114 & 167 & 493 & 62 & 96 & 152 & 88 & 249 & 57 & 49 & 137   \\ \hline

\quad$14$ & 246 & 127 & 131 & 225 & 136 & 128 & 253 & 114 & 164 & 475 & 69 & 87 & 163 & 74 & 235 & 66 & 67 & 119   \\ \hline

\quad$15$ & 257 & 117 & 109 & 265 & 108 & 108 & 249 & 119 & 137 & 453 & 51 & 111 & 177 & 88 & 240 & 68 & 62 & 130    \\ \hline

\quad$16$ & 250 & 121 & 106 & 250 & 112 & 106 & 242 & 108 & 178 & 454 & 66 & 91 & 165 & 81 & 223 & 60 & 68 & 122   \\ \hline

\quad$17$ & 240 & 110 & 147 & 240 & 130 & 119 & 227 & 107 & 155 & 454 & 74 & 107 & 147 & 93 & 248 & 70 & 67 & 138   \\ \hline

\quad$18$ & 235 & 125 & 105 & 227 & 125 & 128 & 266 & 141 & 171 & 496 & 72 & 104 & 147 & 85 & 237 & 81 & 63 & 136     \\ \hline

\ \ Avg & 250 & 121 & 122 & 244 & 121 & 118 & 244 & 120 & 161 & 486 & 70 & 98 & 162 & 82 & 243 & 64 & 62 & 125     \\ \hline
\end{tabular}
}
\end{table}

\begin{table}
{\footnotesize
\caption{Values of $N_{2}(k,D,2,5,10^5)$ for various $k$ and $D$  (see Section \ref{data} for explanations) }\label{2-2}
\begin{tabular}{|p{0.8cm}|p{0.35cm}|p{0.35cm}|p{0.35cm}|p{0.35cm}|p{0.35cm}|p{0.35cm}|p{0.35cm}|p{0.35cm}|p{0.35cm}|p{0.35cm}|p{0.35cm}|p{0.35cm}|p{0.35cm}|p{0.35cm}|p{0.35cm}|p{0.3cm}|p{0.35cm}|p{0.35cm}|}
\hline
\quad $D$ & 13 & 14 & 17 & 22 & 30 & 33 & 51 & 55 & 59 & 67 & 71 & 79 & 83 & 87 & 91 & 95 & 111 & 219   \\ \hline

\quad $I_2$ & 236 & 59 & 118 & 118 & 59 & 118 & 118 & 59 & 79 & 236 & 34 & 47 & 79 & 39 & 118 & 29 & 29 & 59   \\ \hline

$k=3$ & 248 & 56 & 132 & 137 & 60 & 131 & 130 & 68 & 79 & 268 & 38 & 42 & 82 & 40 & 112 & 33 & 30 & 54     \\ \hline

\quad$4$ & 251 & 64 & 119 & 129 & 67 & 116 & 110 & 69 & 103 & 238 & 42 & 53 & 86 & 43 &127 & 30 & 33 & 78     \\ \hline

\quad$5$ & 249 & 70 & 126 & 131 & 55 & 109 & 115 & 63 & 82 & 244 & 31 & 42 & 75 & 45 & 113 & 28 & 27 & 78   \\ \hline

\quad$6$ & 261 & 56 & 104 & 134 & 64 & 106 & 102 & 56 & 86 & 245 & 26 & 52 & 102 & 30 &122 & 37 & 30 & 59     \\ \hline

\quad$7$ & 244 & 65 & 130 & 108 & 60 & 132 & 130 & 55 & 76 & 239 & 45 & 51 & 91 & 44 & 119 & 30 & 24 & 62    \\ \hline

\quad$8$ & 277 & 60 & 128 & 117 & 61 & 116 & 117 & 62 & 62 & 237 & 36 & 40 & 77 & 38 & 120 & 35 & 32 & 50    \\ \hline

\quad$9$ & 264 & 72 & 136 & 113 & 62 & 109 & 126 & 65 & 73 & 266 & 31 & 52 & 91 & 36 & 124 & 26 & 38 & 55     \\ \hline

\quad$10$ & 233 & 64 & 125 & 117 & 67 & 103 & 121 & 47 & 85 & 248 & 26 & 57 & 79 & 30 & 123 & 30 & 22 & 54    \\ \hline

\quad$11$ & 240 & 56 & 126 & 113 & 60 & 135 & 116 & 59 & 77 & 239 & 33 & 52 & 95 & 44 & 119 & 30 & 30 & 48   \\ \hline

\quad$12$ & 243 & 73 & 110 & 131 & 58 & 128 & 108 & 65 & 83 & 250 & 36 & 42 & 87 & 36 & 125 & 32 & 30 & 54   \\ \hline

\quad$13$ & 256 & 62 & 121 & 129 & 53 & 116 & 133 & 61 & 87 & 240 & 31 & 45 & 80 & 32 & 113 & 29 & 24 & 58   \\ \hline

\quad$14$ & 246 & 62 & 131 & 105 & 62 & 128 & 129 & 59 & 79 & 254 & 31 & 40 & 96 & 40 & 129 & 34 & 37 & 65  \\ \hline

\quad$15$ & 257 & 60 & 109 & 132 & 59 & 108 & 124 & 53 & 52 & 233 & 19 & 69 & 86 & 51 & 122 & 33 & 32 & 68    \\ \hline

\quad$16$ & 250 & 63 & 106 & 126 & 56 & 106 & 127 & 55 & 93 & 228 & 29 & 53 & 82 & 53 & 120 & 28 & 28 & 64   \\ \hline

\quad$17$ & 240 & 61 & 147 & 122 & 64 & 119 & 125 & 62 & 80 & 214 & 33 & 50 & 80 & 53 & 128 & 27 & 31 & 68   \\ \hline

\quad$18$ & 235 & 63 & 105 & 112 & 51 & 128 & 141 & 78 & 89 & 249 & 28 & 46 & 73 & 45 & 128 & 37 & 32 & 74     \\ \hline

\ \ Avg & 250 & 63 & 122 & 122 & 60 & 118 & 122 & 61 & 80 & 243 & 32 & 49 & 85 & 41 & 122 & 31 & 30 & 62     \\ \hline
\end{tabular}
}
\end{table}

\begin{table}
{\footnotesize
\caption{Values of $N_{3}(k,D,2,5,10^5)$ for various $k$ and $D$  (see Section \ref{data} for explanations) }\label{2-3}
\begin{tabular}{|p{0.8cm}|p{0.35cm}|p{0.35cm}|p{0.35cm}|p{0.35cm}|p{0.35cm}|p{0.35cm}|p{0.35cm}|p{0.35cm}|p{0.35cm}|p{0.35cm}|p{0.35cm}|p{0.35cm}|p{0.35cm}|p{0.35cm}|p{0.35cm}|p{0.3cm}|p{0.35cm}|p{0.35cm}|}
\hline
\quad $D$ & 13 & 14 & 17 & 22 & 30 & 33 & 51 & 55 & 59 & 67 & 71 & 79 & 83 & 87 & 91 & 95 & 111 & 219   \\ \hline

\quad $I_3$ & 118 & 29 & 59 & 59 & 59 & 118 & 118 & 29 & 39 & 118 & 17 & 24 & 39 & 39 & 59 & 15 & 29 & 59   \\ \hline

$k=3$ & 141 & 32 & 65 & 69 & 60 & 131 & 130 & 35 & 36 & 127 & 19 & 23 & 46 & 40 & 56 & 16 & 30 & 54     \\ \hline

\quad$4$ & 139 & 32 & 70 & 63 & 67 & 116 & 110 & 37 & 50 & 114 & 22 & 23 & 43 & 43 & 56 & 13 & 33 & 78     \\ \hline

\quad$5$ & 127 & 31 & 63 & 64 & 55 & 109 & 115 & 38 & 40 & 119 & 13 & 23 & 40 & 45 & 50 & 20 & 27 & 78   \\ \hline

\quad$6$ & 129 & 32 & 54 & 70 & 64 & 106 & 102 & 29 & 43 & 110 & 11 & 23 & 51 & 30 & 68 & 16 & 30 & 59     \\ \hline

\quad$7$ & 128 & 33 & 62 & 51 & 60 & 132 & 130 & 24 & 33 & 125 & 22 & 26 & 50 & 44 & 68 & 13 & 24 & 62     \\ \hline

\quad$8$ & 130 & 28 & 60 & 67 & 61 & 116 & 117 & 31 & 34 & 115 & 20 & 16 & 33 & 38 & 66 & 18 & 32 & 50     \\ \hline

\quad$9$ & 116 & 32 & 71 & 58 & 62 & 109 & 126 & 28 & 33 & 135 & 15 & 27 & 51 & 36 & 56 & 11 & 38 & 55     \\ \hline

\quad$10$ & 130 & 41 & 61 & 58 & 67 & 103 & 121 & 31 & 43 & 129 & 10 & 27 & 42 & 30 & 61 & 14 & 22 & 54    \\ \hline

\quad$11$ & 110 & 21 & 66 & 56 & 60 & 135 & 116 & 28 & 37 & 120 & 13 & 25 & 43 & 44 & 54 & 14 & 30 & 48   \\ \hline

\quad$12$ & 123 & 38 & 46 & 59 & 58 & 128 & 108 & 35 & 45 & 110 & 16 & 18 & 47 & 36 & 63 & 13 & 30 & 54   \\ \hline

\quad$13$ & 115 & 30 & 58 & 72 & 53 & 116 & 133 & 36 & 49 & 113 & 17 & 20 & 38 & 32 & 52 & 14 & 24 & 58   \\ \hline

\quad$14$ & 115 & 30 & 64 & 60 & 62 & 128 & 129 & 28 & 36 & 139 & 16 & 18 & 45 & 40 & 64 & 21 & 37 & 65   \\ \hline

\quad$15$ & 114 & 41 & 54 & 64 & 59 & 108 & 124 & 30 & 24 & 124 & 8 & 32 & 48 & 51 & 47 & 15 & 32 & 68    \\ \hline

\quad$16$ & 129 & 37 & 58 & 61 & 56 & 106 & 127 & 28 & 52 & 111 & 15 & 31 & 44 &53 & 64 & 9 & 28 & 64   \\ \hline

\quad$17$ & 107 & 38 & 88 & 59 & 64 & 119 & 125 & 26 & 40 & 111 & 20 & 25 & 47 & 53 & 61 & 11 & 31 & 68   \\ \hline

\quad$18$ & 123 & 25 & 53 & 50 & 51 & 128 & 141 & 36 & 48 & 126 & 17 & 17 & 37 & 45 & 69 & 18 & 32 & 74   \\ \hline

\ \ Avg & 124 & 33 & 62 & 61 & 60 & 118 & 122 & 31 & 40 & 121 & 16 & 23 & 44 & 41 & 60 & 15 & 30 & 62     \\ \hline
\end{tabular}
}
\end{table}

\begin{table}
\centering
\caption{Values of $N_{i}(12,3,\rho,10^4,10^8),i=1,2,3$, for various $\rho$ (see Section \ref{data} for explanations)}\label{ideal}
\begin{tabular}{|c|c|c|c|c|c|c|c|c|c|c|}
\hline
$\rho$ & 1.1 & 1.15 & 1.2 & 1.25 & 1.3 & 1.35 & 1.4 & 1.45 & 1.5 & 1.55    \\ \hline

 $I(12,3,\rho,10^4,10^8)$ & 1 & 2  & 4 & 8 & 16 & 32 & 67 & 142 & 304 & 658 \\ \hline

$N_{1}(12,3,\rho,10^4,10^8)$ & 8 & 12 & 15 & 22 & 33 & 47 & 83 & 177 & 355 & 706\\ \hline

$N_{2}(12,3,\rho,10^4,10^8)$ & 2 & 5 & 7 & 11 & 16 & 23 & 43 & 88 & 178 & 388   \\ \hline

$N_{3}(12,3,\rho,10^4,10^8)$ & 2 & 5 & 7 & 11 & 16 & 23 & 43 & 88 & 178 &  388  \\ \hline

\end{tabular}

\end{table}

\section*{Acknowledgement}
The author would like to thank Prof. Igor Shparlinski for introducing him pairing-friendly elliptic curves and providing lots of stimulating suggestions. He is grateful to Prof. John Boxall for his valuable comments and helpful discussions, which play a very important role in improving this paper. He also thanks Dr. Nicolas Mascot and Dr. Aurel Page for teaching him how to use the PlaFRIM.
Finally, he thanks the referee for careful reading and useful comments.


\begin{thebibliography}{20}

\bibitem{Avanzi2005}
R. Avanzi, H. Cohen, C. Doche, G. Frey, T. Lange, K. Nguyen and F. Vercauteren, \emph{Handbook of elliptic
and hyperelliptic curve cryptography}, CRC Press, 2005.

\bibitem{Barreto2006}
P.S.L.M. Barreto and M. Naehrig, \emph{Pairing-friendly elliptic curves of prime order}, in Selected Areas in Cryptography 2005, Lecture Notes in Comput. Sci. \textbf{3897} (2006), 319-331.

\bibitem{Bateman1962}
P.T. Bateman and R.A. Horn, \emph{A heuristic asymptotic formula concerning the distribution of prime numbers}, Math. Comp. \textbf{16} (1962), 363-367.

\bibitem{Boneh2003}
D. Boneh and M. Franklin, \emph{Identity-based encryption from the Weil pairing}, in Crypto 2001, Lecture Notes in Comput. Sci. \textbf{2139} (2001), 213-229. Full version: SIAM J. Comput. \textbf{32} (2003), 586-615.

\bibitem{Boneh2005}
D. Boneh, E.-J. Goh and K. Nissim, \emph{Evaluating 2-DNF formulas on ciphertexts}, in Proceedings of TCC 2005, Lecture Notes in
Comput. Sci. \textbf{3378} (2005), 325-341.

\bibitem{Boneh2004}
D. Boneh, B. Lynn and H. Shacham, \emph{Short signatures from the Weil pairing}, in Asiacrypt 2001, Lecture Notes in Comput. Sci. \textbf{2248} (2001), 514-532. Full version: J. Cryptology \textbf{17} (2004), 297-319.

\bibitem{Boneh2011}
D. Boneh, K. Rubin and A. Silverberg, \emph{Finding composite order ordinary elliptic curves using the
Cocks-Pinch method}, J. Number Theory \textbf{131} (2011), 832-841.

\bibitem{Boxall2012}
J. Boxall, \emph{Heuristics on pairing-friendly elliptic curves}, J. Math. Cryptol. \textbf{6} (2012), 81-104.

\bibitem{Cocks-Pinch}
C. Cocks and R.G.E. Pinch, \emph{Identity-based cryptosystems based on the Weil pairing}, Unpublished manuscript, 2001.


\bibitem{Murty2004}
J. Esmonde and M. Ram Murty, \emph{Problems in algebraic number theory}, GTM 190, Springer-Verlag, 2004.

\bibitem{Finch2010}
S. Finch, G. Martin and P. Sebah, \emph{Roots of unity and nullity modulo $n$}, Proc. Amer. Math. Soc. \textbf{138} (2010), 2729-2743.

\bibitem{Freeman2010}
D. Freeman, M. Scott and E. Teske, \emph{A taxonomy of pairing-friendly elliptic curves}, J. Cryptology \textbf{23} (2010), 224-280.

\bibitem{Frey1994}
G. Frey and H. R{\"u}ck, \emph{A remark concerning $m$-divisibility and the discrete logarithm in the divisor class
group of curves}, Math. Comp. \textbf{62} (1994), 865-874.

\bibitem{Galbraith2005}
S. Galbraith, \emph{Pairings}, Ch. IX of I. Blake, G. Seroussi, N. Smart (Eds.), Advances in Elliptic Curve Cryptography, Cambridge University Press, Cambridge, 2005.

\bibitem{Hardy}
G.H. Hardy and E.M. Wright, \emph{An introduction to the theory of numbers}, Oxford University Press, Oxford, 1979.

\bibitem{Hayashi}
T. Hayashi, T. Shimoyama, N. Shinohara and T. Takagi, \emph{Breaking pairing-based
cryptosystems using $\eta_{T}$ pairing over $GF(3^{97})$}, in Asiacrypt 2012, Lecture Notes in Comput. Sci. \textbf{7658} (2012), 43-60.

\bibitem{Joux2000}
A. Joux, \emph{A one round protocol for tripartite Diffie-Hellman}, in Algorithmic Number Theory
Symposium 2000, Lecture Notes in Comput. Sci. \textbf{1838} (2000), 385-393.

\bibitem{Koblitz2005}
N. Koblitz and A.J. Menezes, \emph{Pairing-based cryptography at high security levels}, LNCS \textbf{3796} (2005), 13-36.

\bibitem{Korevaar2010}
J. Korevaar and H. Te Riele, \emph{Average prime-pair counting formula}, Math. Comp. \textbf{79} (2010), 1209-1229.

\bibitem{Shparlinski2006}
F. Luca and I.E. Shparlinski, \emph{Elliptic curves with low embedding degree}, J. Cryptology \textbf{19} (2006), 553-562.


\bibitem{Menezes1993}
A. Menezes, T. Okamoto and S. Vanstone, \emph{Reducing elliptic curve logarithms to logarithms in a finite field},
IEEE Trans. Inform. Theory \textbf{39} (1993), 1639-1646.

\bibitem{Miller2004}
V. Miller, \emph{The Weil pairing, and its efficient calculation}, J. Cryptology \textbf{17} (2004) 235-261.

\bibitem{Miyaji2001}
A. Miyaji, M. Nakabayashi and S. Takano, \emph{New explicit conditions of elliptic curve traces for FR-reduction},
IEICE Trans. Fundam. \textbf{E84-A} (2001), 1234-1243.

\bibitem{Narkiewicz2004}
W. Narkiewicz, \emph{Elementary and analytic theory of algebraic numbers}, Springer-Verlag, 2004.

\bibitem{Pari}
PARI/GP, version {\tt 2.5.3}, Bordeaux, 2012, \url{http://pari.math.u-bordeaux.fr/}.

\bibitem{Paterson2005}
K. Paterson, \emph{Cryptography from pairings}, Ch. X of I. Blake, G. Seroussi and N. Smart (Eds.), Advances in Elliptic Curve Cryptography, Cambridge University Press, 2005.

\bibitem{Sakai2000}
R. Sakai, K. Ohgishi and M. Kasahara, \emph{Cryptosystems based on pairing}, in Symposium on Cryptography
and Information Security 2000, Okinawa, Japan, 2000.

\bibitem{Sutherland2011}
A.V. Sutherland, \emph{Computing Hilbert class polynomials with the Chinese Remainder Theorem}, Math. Comp. \textbf{80} (2011), 501-538.

\bibitem{Shparlinski2012}
J.J. Urroz, F. Luca and I.E. Shparlinski, \emph{On the number of isogeny classes and pairing-friendly elliptic
curves and statistics for MNT curves}, Math. Comp. \textbf{81} (2012), 1093-1110.

\bibitem{Verheul2001}
E. Verheul, \emph{Evidence that XTR is more secure than supersingular elliptic curve cryptosystems}, in Eurocrypt 2001, Lecture Notes in Comput. Sci.
\textbf{2045} (2001), 195-210. Full version: J. Cryptology \textbf{17} (2004), 277-296.


\end{thebibliography}
\end{document}